\documentclass{article}
\usepackage{pgfplots}
\usepackage{fullpage}
\usepackage{amsmath}
\usepackage{amsthm}
\usepackage{amssymb}
\usepackage{url}
\usepackage{graphicx}
\usepackage{verbatim}
\usepackage{url}
\usepackage{graphicx}
\usepackage[caption=false]{subfig}
\usepackage{float}
\usepackage[ruled]{algorithm2e}
\usepackage{tikz}
\usetikzlibrary{calc}
\usetikzlibrary{patterns}
\usepackage{tikz}
\tikzstyle{mybox} = [draw=black, fill=white,  thick,
rectangle, inner sep=10pt, inner ysep=20pt]
\tikzstyle{mybox} = [draw=black, fill=white,  thick,
rectangle, inner sep=2pt, inner ysep=2pt]
\usetikzlibrary{arrows}
\usetikzlibrary{arrows,chains,matrix,positioning,scopes}
\usetikzlibrary{arrows,shapes,positioning}
\usetikzlibrary{decorations.markings}
\tikzstyle arrowstyle=[scale=1]
\tikzstyle directed=[postaction={decorate,decoration={markings,
		mark=at position .65 with {\arrow[arrowstyle]{stealth}}}}]
\tikzstyle reverse directed=[postaction={decorate,decoration={markings,
		mark=at position .65 with {\arrowreversed[arrowstyle]{stealth};}}}]
\newcommand{\boundellipse}[3]
{(#1) ellipse (#2 and #3)
}

\newtheorem{theorem}{Theorem}

\newtheorem{cor}{Corollary}
\newtheorem{prop}{Proposition}
\theoremstyle{definition}
\newtheorem{definition}{Definition}
\newtheorem{remark}{Remark}

\usepackage{multirow}
\usepackage[utf8]{inputenc}

\begin{document}

\title{A Geometric Algorithm for Solving Linear Systems}
\author{Bahman Kalantari, Chun Lau, Yikai Zhang \\
Department of Computer Science, Rutgers University, NJ\\
kalantari@cs.rutgers.edu, larryl@cs.rutgers.edu, yz422@cs.rutgers.edu
}
\date{}
\maketitle

\begin{abstract}
Based on the geometric {\it Triangle Algorithm} for testing membership of a point in a convex set, we present a novel iterative algorithm for testing the solvability of a real linear system $Ax=b$, where $A$ is an $m \times n$  matrix of arbitrary rank. Let $C_{A,r}$ be the ellipsoid determined as the image of the Euclidean ball of radius $r$ under the linear map $A$. The basic procedure in our  algorithm computes a point in $C_{A,r}$ that is either within $\varepsilon$ distance to $b$, or acts as a certificate proving $b \not \in C_{A,r}$.  Each iteration takes $O(mn)$ operations and when $b$ is well-situated in $C_{A,r}$, the number of iterations is proportional to $\log{(1/\varepsilon)}$. If $Ax=b$ is solvable the algorithm computes an approximate solution or the minimum-norm solution. Otherwise, it computes a certificate to unsolvability, or the minimum-norm least-squares solution. It is also applicable to complex input.  In a computational comparison with the state-of-the-art algorithm BiCGSTAB ({\it Bi-conjugate gradient method stabilized}), the Triangle Algorithm is very competitive. In fact, when the iterates of BiCGSTAB do not converge, our algorithm can verify $Ax=b$ is unsolvable and approximate the minimum-norm least-squares solution. The Triangle Algorithm is robust, simple to implement, and requires no preconditioner, making it attractive to practitioners, as well as researchers and educators.
\end{abstract}

\section{Introduction}
The significance of solving a linear system of equations arises in a variety of areas, such as Numerical Analysis, Economics, Computational Biology, and even in high school education.  Solving a linear system of equations is undoubtedly one of the most practical problems in numerous aspects of scientific computing. Gaussian elimination is the most familiar method for solving a linear system of equations, discussed in numerous books, e.g. Atkinson \cite{Atkinson}, Bini and Pan \cite{Bini}, Golub and van Loan \cite{Golub}, and Strang \cite{Strang}.  The method itself is an important motivation behind the study of linear algebra. Iterative methods for solving linear systems offer very important alternatives to direct methods and find  applications in problems that require the solution of very large or sparse linear systems. For example, problems from discretized partial differential equations lead to large sparse systems of equations, where direct methods become impractical.

Iterative methods, when applicable, generate a sequence of approximate solutions. They begin with an initial approximation and successively improve it until a desired approximation is reached. Iterative methods include such classical methods as the Jacobi, the Gauss-Seidel, the successive over-relaxation (SOR), the accelerated  over-relaxation (AOR), and the symmetric successive over-relaxation method which applies when the coefficient matrix is symmetric. When the coefficient matrix is symmetric and positive definite the conjugate gradient method (CG) becomes applicable. Convergence rate of iterative methods can often be substantially accelerated by preconditioning.  Some major references in the vast subject of iterative methods include, Barrett et al. \cite{Bar}, Golub and van Loan \cite{Golub},  Greenbaum \cite{Green}, Hadjidimos \cite{Hadjid}, Saad \cite{Saad}, van der Vorst \cite{van1}, Varga \cite{Var}, and Young \cite{Young}.

To guarantee convergence, iterative methods often require the coefficient matrix to satisfy certain conditions.  Also, certain decompositions are necessary to carry out some of these iterative methods. In some earlier analysis of iterative methods only convergence criteria are considered, rather than algorithmic complexity. However, in some cases theoretical complexity analysis is provided, see e.g. Reif \cite{Reif} who considers the complexity of iterative methods for sparse diagonally dominant matrices. The steepest descent method and conjugate gradient method are well-analyzed for solving an  $n \times n$ linear system $Ax = b$ where $A$ is symmetric positive definite. When $A$ is not positive definite one can apply these methods to the equivalent system $A^TAx=A^Tb$. However, a general consensus is that solving the normal equations can be inefficient when $A$ is poorly conditioned, see e.g. Saad \cite{Saad}. The reason being that the condition number of $A^TA$ is the square of the condition number of $A$.

The major computational effort in each iteration of the iterative methods involves matrix-vector multiplication, thus requiring $O(n^2)$ operations. This makes iterative methods very attractive for solving large systems and also for parallelization. There is also a vast literature on parallelization of iterative methods for large systems, see Demmel \cite{Dem}, Dongarra et al. \cite{Don}, Duff and van der Vorst \cite{Duff},  van der Vorst  \cite{van1}, van der Vorst and Chan \cite{Van2}. A popular iterative method in practice to solve a general (non-symmetric) linear system is the {\it bi-conjugate gradient method stabilized} (BiCGSTAB), a Krylov subspace method that has faster and smoother convergence than {\it conjugate gradient squared method (CGS)} as well as the {\it biconjugate gradient method}. More details are given in van der Vorst \cite{Vorst1} and Saad \cite{Saad}. BiCGSTAB is the state-of-the-art method of choice as well \cite{Chen}.

In this article we consider solving a linear system $Ax=b$, where $A$ is an $m \times n$ real or complex matrix of arbitrary rank and describe a novel geometric algorithm for computing an approximate solution, minimum-norm solution, least-squares solution, or minimum-norm least-squares solution. Thus, when $A$ is invertible, our algorithm approximates the unique solution. When $A$ is not invertible but $Ax=b$ is solvable, it approximates the solution with smallest norm, and when $Ax=b$ has no solution, it approximates the solution that minimizes $\Vert Ax - b \Vert$ with smallest norm. Also, when $Ax=b$ has no solution, it computes a certificate proving unsolvability.

The algorithm is inspired by the {\it Triangle Algorithm} developed for the {\it  convex hull membership problem} problem and its generalization in Kalantari \cite{kalfull, kalcon}.  Driven by the convex hull membership problem,  we propose a variant of the Triangle Algorithm for linear systems of arbitrary dimension, requiring no assumption on the matrix $A$. We also make a computational comparison
with the state-of-the-art algorithm BiCGSTAB in solving a square linear system. Our algorithm shows promising and strong results in this comparison, both in terms of speed and the quality of solutions. The relevance of this comparison lies in the fact that both algorithms offer alternatives to exact methods and to the iterative methods while requiring none of the structural restrictions of the latter methods. The extreme simplicity and theoretical complexity bounds suggest practicality of the new method, especially for large scale or sparse systems. Our computational results are very encouraging and support the practicality of the Triangle Algorithm.

The article is organized as follows. In Section \ref{sec2}, we describe the basics of the Triangle Algorithm and its complexity bounds.  In Section \ref{sec3}, we consider the Triangle Algorithm and its modifications for solving a linear system, $Ax=b$. We prove some properties for the basic tasks in the Triangle Algorithm in testing if $b$ lies in an ellipsoid $C_{A,r}=\{Ax: \Vert x \Vert \leq r\}$, i.e.  the image of ball of radius $r$ under the linear map $A$. In  Subsection \ref{subsec3.1}, we describe Algorithm \ref{alg1} for testing if $b \in C_{A,r}$. In Subsection \ref{subsec3.2}, we describe Algorithm \ref{alg2} for testing if $Ax=b$ is solvable, where the radius $r$ is repeatedly increased, starting from an initial estimate. In Subsection \ref{subsec3.3}, we establish properties of this algorithm for testing the solvability of $Ax=b$ and if not, in approximating the least-squares solution, as well as the minimum-norm least-squares solution. In Section \ref{sec4}, we present computational results with the Triangle Algorithm and contrast it with the widely-used BiCGSTAB. We conclude with some final remarks.

\section{Triangle Algorithm for General Convex Hull Membership} \label{sec2}
Let the {\it general convex hull membership} (GCHM) problem be defined as follows:
Given a  bounded subset $S$ in $\mathbb{R}^m$ and a distinguished point $p \in \mathbb{R}^m$, determine if $p \in C=conv(S)$, the convex hull of $S$. Specifically, given $\varepsilon \in (0, 1)$, either compute $p' \in C$ such that $\|p' - p\| \leq \varepsilon$, or find a hyperplane that separates $p$ from $C$.  We describe the basics of the {\it Triangle Algorithm} for solving GCHM from \cite{kalfull, kalcon}.

\begin{definition} \label{defp} Given an arbitrary  $p' \in C$, $p' \not =p$,  called {\it iterate}, a point $v \in C$ is called a {\it pivot} (at $p'$) if
\begin{equation} \label{pivot}
\|p' - v\| \geq \|p - v\| \quad \iff \quad
(p-p')^Tv \geq \frac{1}{2} (\Vert p \Vert^2 - \Vert p' \Vert^2).
\end{equation}
If no pivot exists $p'$ is called a {\it witness}. A pivot $v$ is a {\it strict pivot} if
\begin{equation} \label{pivots}
(p-p')^T(v-p) \geq 0.
\end{equation}
Thus, when the three points are distinct $\angle p'pv$ is at least $\pi/2$.
\end{definition}

\begin{prop} An iterate $p' \in C$ is a witness if and only if
the orthogonal bisecting hyperplane to the line segment $pp'$ separates $p$ from $C$, implying $p \not\in C$. \qed
\end{prop}

The Triangle Algorithm works as follows: Given $\varepsilon \in (0,1)$ and $p' \in C$, if $\Vert p - p' \Vert \leq \varepsilon$ it stops. If $p'$ is a witness, $p \not \in C$. Otherwise, it computes a pivot $v$ and computes the next iterate $p''$ as the nearest point to $p$ on the line segment $p'v$.

\begin{prop} \label{prop1}  Given $p$, $p'$ and a pivot $v$, the nearest point of $p$ on $p'v$ is
\begin{equation} \label{alph}
    p'' = (1 - \alpha)p' + \alpha v, \quad  \alpha = \min \bigg \{1, \frac{(p - p')^T(v - p')}{\|v - p'\|^2} \bigg \}. \qed
\end{equation}
\end{prop}

The Triangle Algorithm replaces $p'$ with $p''$ and repeats the above iteration. The test for the existence of a pivot (or strict pivot) in the worst-case amounts to computing the optimal solution $v_*$ of
\begin{equation} \label{pivottest}
\max\{c^Tx : x \in C\}, \quad c = p - p'.
\end{equation}
It follows from (\ref{pivot}) that $v_*$ is either a pivot, or $p'$ is a witness. In fact, if $p \in C$, it follows from (\ref{pivots}) that $v_*$ is a strict pivot. The correctness of the Triangle Algorithm for GCHM is due to the following theorem proved in \cite{kalfull, kalcon}.
\begin{theorem} \label{thm1} {\rm {(Distance Duality)}}
$p \in C$ if and only if for each $p' \in C$ there exists a (strict) pivot $v \in C$. Equivalently, $p \not\in C$ if and only if there exists a witness $p' \in C$.
\qed
\end{theorem}
The iteration complexity for Triangle Algorithm is given in the following theorem.
\begin{theorem} \label{thm2} {\rm {(Complexity Bounds)}}
Let $R = \max\{\|x - p\| : x \in C\}$. Let $\varepsilon \in (0,1)$.  Triangle Algorithm terminates with  $p' \in C$ such that one of the following conditions is satisfied:

(i) $\|p - p_{\varepsilon}\| \leq \varepsilon$ and the number of iterations is $O(R^2/\varepsilon^2)$. If $p$ is the center of ball of radius $\rho$ in the relative interior of $C$ and each iteration uses a strict pivot, the number of iterations is $O((R/\rho)^2 \ln{1 / \varepsilon})$.

(ii) $p'$ is a witness and the number of iterations is $O(R^2/\delta_*^2)$, where $\delta_* = \min \{\|x - p\| : x \in C\}$. Moreover,
\[
\delta_* \leq \|p' - p\| \leq 2\delta_*. \tag*{\qed}
\]
\end{theorem}

\begin{remark}
When $p$ is interior to $C$ we can think of the ratio $R/\rho$ as a {\it condition number} for the problem. If this condition number is not large the complexity is only logarithmic in $1/\varepsilon$, hence few iterations will suffice.
\end{remark}
More details on the Triangle Algorithm for GCHM is given in \cite{kalfull, kalcon}. For applications of Triangle Algorithm in optimization, computational geometry, and machine learning, see \cite{AKZ}.

\section{Application of Triangle Algorithm to Linear Systems} \label{sec3}
Given an $m \times n$ real matrix $A$, and $b \in \mathbb{R}^m$, we wish to solve
\begin{equation} \label{eqnn}
Ax = b.
\end{equation}
We may also wish to solve the {\it normal equation}:
\begin{equation} \label{eqnnm}
A^TAx =A^Tb.
\end{equation}
Additionally, it may be desirable to compute the {\it minimum-norm} solution of (\ref{eqnn}) or (\ref{eqnnm}). While (\ref{eqnn}) may be unsolvable, (\ref{eqnnm}) is always solvable. A direct proof of solvability of (\ref{eqnnm}) is as follows: From one of the  many equivalent formulations of the well-known Farkas Lemma,  described in every linear programming book,  it follows that if $A^TAx=A^Tb$ is not solvable there exists $y$ such that $A^TAy=0$, $b^TAy <0$.  But this gives $y^TA^TAy= \Vert A y \Vert ^2=0$. Hence $Ay=0$, contradicting that $b^TAy <0$.

Any solution to the normal equation satisfies the least-squares formula
\begin{equation} \label{deltax}
\Delta= \min \{\Vert Ax - b \Vert : x \in \mathbb{R}^n \}.
\end{equation}
Hence, a solution to the normal equation is a  {\it least-squares} solution. Thus once we have a solution $x$ to the normal equation we can check the solvability of $Ax=b$.  The {\it minimum-norm} least-squares solution, denoted by $x_*$, is the solution to (\ref{deltax}) with minimum norm. It is known that $x_*=V\Sigma^{\dagger}U^Tb$,
where $U \Sigma V^T$ is the singular value decomposition of $A$, and $\Sigma^{\dagger}$ is the pseudo-inverse of $\Sigma$. In particular, if $\sigma_*$ is the least singular value,
\begin{equation} \label{rstar}
\Vert x_* \Vert \leq  \frac{1}{ \sigma_*} \Vert b \Vert.
\end{equation}

In this article we will develop a version of the Triangle Algorithm that can solve the approximate versions of (\ref{eqnn}) or (\ref{eqnnm}) in an interrelated fashion. The algorithm is simple to implement and, as our computational results demonstrate, it works very well in practice.

Given $r >0$, consider the ellipsoid defined as the image of the ball of radius $r$ under the linear map $A$:
\begin{align}
C_{A,r} = \{y = Ax: \|x\| \leq r\}.
\end{align}
We will analyze the Triangle Algorithm  for testing if $b \in C_{A,r}$.

\begin{prop} \label{prop0} Let $p=b$.  Given $x' \in \mathbb{R}^n$, where $\|x'\| \leq r$, let $p' = Ax'$. Assume $p' \not =p$. Let
\begin{equation} \label{c}
c= A^T(p-p')=A^TAx'-A^Tb.
\end{equation}

(1) If $c=0$, $p'$ is witness.

(2) If $c \not =0$, let $v_r=rA{c}/{\|c\|}$. Then
\begin{equation}  \label{vr}
\max\{c^Tx: \|x\| \leq r\} = c^T v_r= r \Vert c \Vert.
\end{equation}

(3)  $v_r$ is a pivot if and only if
\begin{equation} \label{eqlem}
r \Vert c \Vert \geq \frac{1}{2} (\Vert p \Vert^2 - \Vert p' \Vert^2).
\end{equation}

(4) $v_r$ is a strict pivot if and only if
\begin{equation}  \label{eqlem2}
r \Vert c \Vert \geq (p-p')^Tp.
\end{equation}
Furthermore, if $p \in C_{A,r}$, $v_r$ is a strict pivot. On the other hand, if  $v_r$ is not a strict pivot, the orthogonal hyperplane to the line segment  $p'p$, passing to through the nearest point of $v_r$, separates $p$ from $C_{A,r}$.
\end{prop}
\begin{proof} (1):  Since $p' \not = p$,  if $p$ is not a witness, from the Triangle Algorithm  there exists a pivot in $C_{A,r}$ and this in turn implies there exists  $p'' \in C_{A,r}$ such that $\Vert p''-p \Vert < \Vert p'-p \Vert$. But since $x'$ is a solution to the normal equation $\Vert p'-p \Vert$ is minimum, a contradiction.

(2): The first equality in (\ref{vr}) follows trivially since the optimization is over a ball. The second equality follows from the definition of $v_r$.

(3) and (4): Considering the definitions of pivot and strict pivot in (\ref{pivots}),  (\ref{eqlem}) and (\ref{eqlem2}) follow. The proof of the remaining part is analogous to proving that a witness induces a separating hyperplane.
\end{proof}

\begin{remark} The computation of a pivot in $C_{A,r}$ can be established efficiently, in $O(mn)$ operations. Geometrically, to find $v_r$, consider the orthogonal hyperplane to line $pp'$, then move the hyperplane from $p'$ toward $p$ until it is tangential to the boundary of the ellipsoid. For illustration, see Figure \ref{AAD}. Before describing  the algorithm to test if $p$ lies in $C_{A,r}$ we state some results.
\end{remark}

\begin{theorem}  \label{new} $Ax=b$ is unsolvable if and only if for each $r>0$, there exists a witness $p' \in C_{A,r}$, dependent on $r$, such that if $p=b$, $c=A^T(p-p')$, then
\begin{equation} \label{infeasibility}
r \Vert c \Vert < \frac{1}{2} (\Vert p \Vert^2 - \Vert p' \Vert^2) \leq (p-p')^Tp.
\end{equation}
\end{theorem}

\begin{proof} Suppose $Ax=b$ is not solvable. Given, $r >0$, $p \not \in C_{A,r}$. Hence, there exists a witness $p' \in C_{A,r}$ and by (\ref{eqlem}) in Proposition \ref{prop0}, the strict inequality in (\ref{infeasibility}) holds. But then the  inequality holds from the fact that $\Vert p - p' \Vert ^2 \geq 0$.
Conversely, if for each $r>0$, there exists $p'$ such that (\ref{infeasibility}) holds, then $p'$ is witness, implying $b \not \in C_{A,r}$. Hence, $Ax=b$ is unsolvable.
\end{proof}

\begin{cor} \label{cor1} Suppose $p=b \not \in C_{A,r}$. Then there exists $p' \in C_{A,r}$ such that $p \not \in C_{A,r'}$, for all $r'$ satisfying
\begin{equation}
r \leq  r' < \frac{(p-p')^Tp}{ \Vert c \Vert}.
\end{equation}
\end{cor}
\begin{proof}  The result follows immediately from the previous theorem.
\end{proof}

\begin{remark} Corollary \ref{cor1} suggests a witness in
$C_{A,r}$ can be used to compute a lower bound to the norm of the minimum-norm solution to $Ax=b$, hence a lower bound to $\Vert x_* \Vert$ (see (\ref{rstar})).
\end{remark}

\subsection{Triangle Algorithm for Testing if $b$ Lies in the Ellipsoid $C_{A,r}$} \label{subsec3.1}

Algorithm \ref{alg1} below describes the Triangle Algorithm for testing if $b$ lies in $C_{A,r}$.

\begin{algorithm}[!htb] \label{alg1}
\SetAlgoNoLine
\KwIn{$m \times n$ matrix $A$, $b \in \mathbb{R}^m$, $b \not =0$, $r>0$, $\varepsilon \in (0, 1)$.}
 $p \gets b$, $x' \gets 0$, $p' \gets 0$ \\
 \While{$\|p - p'\| > \varepsilon$ and $c=A^T(p-p') \not =0$}{ $v_r=
A\frac{c}{\|c\|}$. \\
  \lIf{$ r \Vert c \Vert \geq \frac{1}{2} (\Vert p \Vert ^2 - \Vert p' \Vert^2)$} { $\alpha = \min \{ 1, (p - p')^T(v_r - p')/\|v_r - p'\|^2\}$, \quad $p' \gets (1-\alpha) p' + \alpha v_r$, \quad $x' \gets (1-\alpha) x' +  \alpha \frac{rc}{\Vert c \Vert}$}
   \Else{{STOP}}}
    \caption{Testing if $b$ lies in $C_{A,r}$}
\end{algorithm}

Taking into account that when $p \in C_{A,r}$, if $v_r$ is a pivot it is also a strict pivot, as well as using Proposition \ref{prop0} and the bound, a restatement of Theorem \ref{thm2}  for Algorithm \ref{alg1} that tests if $p=b \in C_{A,r}$ for a given $r$ is the following:
\begin{equation}
\max \{ \Vert Ax-b \Vert : \Vert x \Vert \leq r\} \leq 2\max \{ \Vert Ax\Vert : \Vert x \Vert \leq r\}= 2r \Vert A \Vert.
\end{equation}

\begin{theorem} \label{thm2p} {\rm {(Complexity Bounds)}}
Algorithm \ref{alg1} terminates with  $p' \in C_{A,r}$ such that one of the following conditions is satisfied and where each iteration takes $O(mn)$ operations:

(i) $\|p - p'\| \leq \varepsilon$ and the number of iterations is $O(r^2 \Vert A \Vert^2/\varepsilon^2)$. If $p$ is the center of ball of radius $\rho$ in the relative interior of $C_{A,r}$, the number of iterations is $O((r \Vert A \Vert/\rho)^2 \ln{1 / \varepsilon})$.

(ii) $p'$ is witness and the number of iterations is $O(r^2 \Vert A \Vert^2/\delta_r^2)$, where $\delta_r = \min \{\|Ax - p\| :  \Vert x \Vert \leq r\}$. Moreover,
\begin{equation} \label{deltabound}
\delta_r \leq \|p - p'\| \leq 2\delta_r. \qed
\end{equation}

\end{theorem}

When $p \in C_{A,r}$ the radius $r$ plays a role in the complexity of the algorithm. This is shown in Figure \ref{tfig:1}.  Figure \ref{AAD} shows one iteration of the algorithm. If for some $x$ with $\Vert x \Vert \leq r$ we have $Ax=b$, then the final $x'$ is an approximate solution. Otherwise, it follows that $Ax=b$ is not solvable when $\Vert x \Vert \leq r$.

\begin{figure}[htpb]
	\centering
	\subfloat[]{
	\raisebox{6.5ex}{
	\begin{tikzpicture}[scale=1]
	\begin{scope}
	\begin{scope}
	[rotate=47.5829]
	\draw  \boundellipse{0,0}{2.810}{1.210};
	\filldraw (0,0) circle (.5pt) node[below] {$O$};
	\end{scope}
	\filldraw (1.2,1.9) circle (.5pt)  node[below] {$p$};
	
	\begin{scope}
	[rotate=0]
	\draw  \boundellipse{1.2,1.9}{0.31}{0.31};
	\end{scope}
	
	\end{scope}
	
	\filldraw (0,0) circle (.5pt);	
	
	\end{tikzpicture}}}
	~
	\subfloat[]{
	\begin{tikzpicture}[scale=1]
	\begin{scope}
	\begin{scope}
	[rotate=47.5829]
	\draw  \boundellipse{0,0}{4.215}{1.815};
	\filldraw (0,0) circle (.5pt) node[below] {$O$};
	\end{scope}
	\filldraw (1.2,1.9) circle (.5pt)  node[below] {$p$};
	
	\begin{scope}
	[rotate=0]
	\draw  \boundellipse{1.2,1.9}{1.1}{1.1};
	\end{scope}
	
	\end{scope}
	
	\filldraw (0,0) circle (.5pt);	
	
	\end{tikzpicture}}
	
	\caption{The ellipsoid $C_{A,r}$ contains $p$ in its interior.  However, by increasing the radius $r$ to $2r$ the condition number of the problem improves significantly.}
	\label{tfig:1}
\end{figure}
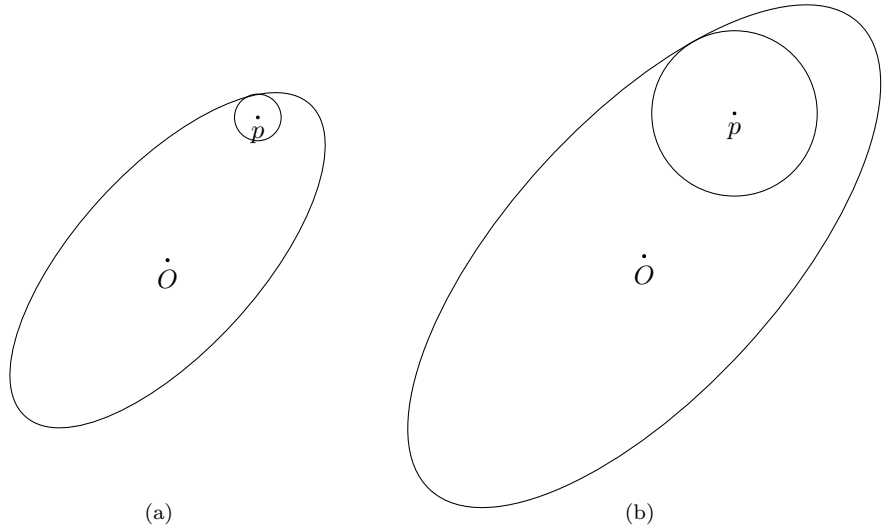

\begin{figure}[htpb]
	\centering
	\subfloat[] {
	\raisebox{10ex}{
	\begin{tikzpicture}[scale=.7]
	\begin{scope}
	\begin{scope}
	[rotate=47.5829]
	\draw \boundellipse{(0,0}{2.810}{1.210};
	\filldraw (0,0) circle (.5pt) node[below] {$O$};
	\end{scope}
	\filldraw (0,2.79) circle (.5pt)  node[above] {$p$};
	\filldraw (0.9815,2.158)  circle (.5pt)  node[above] {$v_r$};
	\filldraw (0.6580,0.75)  circle (.5pt)  node[left] {$p'$};
	\draw[dashed] (-1.37,1.4)--(2.69,2.709);
	\draw (0.6580,0.75)--(0.9815,2.158)--(0,2.79)--(0.6580,0.75);
	\end{scope}
	\label{1st}
	\end{tikzpicture}}}
	\qquad
	\subfloat[] {
	\begin{tikzpicture}[scale=.7]
	\begin{scope}
	\begin{scope}
	[rotate=47.5829]
	\draw  \boundellipse{(0,0}{5.620}{2.420};
	\filldraw (0,0) circle (.5pt) node[below] {$O$};
	\end{scope}
	\filldraw (0,2.79) circle (.5pt)  node[below] {$p$};
	\filldraw (1.963,4.315)  circle (.5pt)  node[above] {$v_r$};
	\filldraw (1.234,2.336)  circle (.5pt)  node[right] {$p''$};
	\filldraw (0.6580,0.75)  circle (.5pt)  node[left] {$p'$};
	\draw[dashed] (-1.34,3.25)--(3.9,4.94);
	\draw (1.963,4.315)--(1.234,2.336)--(0.6580,0.75)--(0,2.79)--(1.963,4.315);
	\draw (0,2.79)--(1.234,2.336);
	\end{scope}
	
	\end{tikzpicture}
	\label{2nd}
	}
	
	\caption{In the left-hand-side Figure \ref{1st} $v_r$ is not a strict pivot, proving $p$ is exterior to the ellipsoid. However, Algorithm \ref{alg1} terminates when $p'$ is a witness with a better estimate of the distance from $p$ to $C_{A,r}$.
In Figure \ref{2nd}  $p'$ admits a pivot $v_r$, used to compute the next iterate $p''$  as the nearest point to $p$ on the line segment $p'v_r$. The pivot $v_r$ is found by moving the orthogonal hyperplane to line $pp'$ (dashed line) in the direction from $p'$ to $p$ until the hyperplane is tangential to the ellipsoid.}
	\label{AAD}
\end{figure}
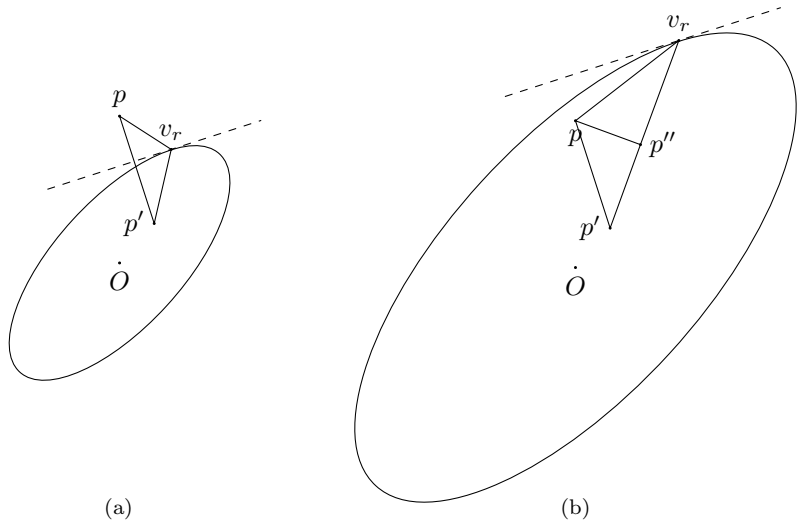

\subsection{Triangle Algorithm for Testing Solvability of $Ax=b$} \label{subsec3.2}
Given a linear system $Ax=b$, to test if it is solvable  we need a more elaborate algorithm than Algorithm \ref{alg1}. The algorithm below computes an $\varepsilon$-{\it approximate} solution $x_\varepsilon$ satisfying  $\Vert A x_\varepsilon - b \Vert \leq \varepsilon$,  when such a solution exists.  Given an initial radius $r_0>0$, it checks if an approximate solution exists in $C_{A,r_0}$. When a  pivot does not exists, it at least doubles the radius $r_0$ (see Corollary \ref{cor1}) and repeats the process.

\begin{algorithm}[!htb] \label{alg2}
\SetAlgoNoLine
\KwIn{$m \times n$ matrix $A$, $b \in \mathbb{R}^m$, $b \not =0$,  $r_0>0$, $\varepsilon
\in (0, 1)$.}
  $p \gets b$, $x' \gets 0$,  $p' \gets  0$, $r \gets r_0$ \\
  \While{$\|p - p'\| > \varepsilon$ and $c=A^T(p-p') \not =0$}{$v_r=
A\frac{c}{\|c\|}$. \\
  \lIf{$ r \Vert c \Vert \geq \frac{1}{2} ( \Vert p \Vert^2 - \Vert p' \Vert^2)$} { $\alpha = min \{1, (p - p')^T(v_r - p')/\|v_r - p'\|^2\}$, \quad $p' \gets (1-\alpha) p' + \alpha v_r$, \quad $x' \gets (1-\alpha) x' +  \alpha \frac{rc}{\Vert c \Vert}$}
   \Else{{$r \gets \max \{\frac{(p-p')^Tp}{||c||}, 2r\}$}}
   }
  \caption{Triangle Algorithm for Solving $Ax = b$}
\end{algorithm}

\subsection{Properties of Algorithm \ref{alg2}} \label{subsec3.3}
Here we prove properties of Algorithm \ref{alg2} showing that while it is designed to solve $Ax=b$, it can also be used to solve the normal equation and detect the unsolvability of $Ax=b$.

\begin{theorem}  \label{finalthm} Algorithm \ref{alg2} satisfies the following properties:

(i)  If $Ax=b$ is solvable, given any $r \geq \Vert x_* \Vert$, Algorithm \ref{alg2} computes  $x'$, $\Vert x' \Vert \leq r$ such that
\begin{equation} \label{eqaxb}
\Vert A x' - b \Vert \leq \varepsilon.
\end{equation}
Furthermore, if $b \not \in C_{A,r_0}$,
\begin{equation} \label{eqaxb1}
\Vert x'  \Vert \leq 2 \Vert x_* \Vert.
\end{equation}

(ii)  Let  $r_\varepsilon =  {\Vert b \Vert^2}/{\varepsilon}$.  Given any $r \geq r_\varepsilon$, Algorithm \ref{alg2} computes  $x'$, $\Vert x' \Vert \leq r_\varepsilon$ such that it either satisfies (\ref{eqaxb}), or it satisfies (\ref{eqi}):
\begin{equation}  \label{eqi}
\Vert  A^TAx' - A^Tb\Vert  \leq \varepsilon.
\end{equation}

(iii) Let $r'_\varepsilon= (\Vert b \Vert / \varepsilon) \max  \big \{\Vert b \Vert,  {2}/{\sigma_*} \big \}$. Given any $r \geq r'_\varepsilon$, Algorithm \ref{alg2} computes $x'$ with $\Vert x' \Vert \leq r_\varepsilon$ such that it either satisfies (\ref{eqaxb}) or it satisfies (\ref{eqii}):
\begin{equation} \label{eqii}
\Vert A^TAx'- A^Tb \Vert \leq \varepsilon, \quad \bigg | (p-p')^Tp  - \Vert p - p' \Vert^2 \bigg | \leq \varepsilon, \quad
4 \Delta^2  \geq  (p-p')^Tp  - \varepsilon.
\end{equation}
In particular, if $(p-p')^Tp \geq 2 \varepsilon$,  then $\Delta \geq \sqrt{\varepsilon}/2$ (hence $Ax=b$ is unsolvable).
\end{theorem}
\begin{proof}

Proof of (i): In this case $Ax_*=b$. Hence for any $r \geq \Vert x_* \Vert$,
$b \in C_{A,r}$. Proof of (\ref{eqaxb1}) follows from the way the value of $r$ is increased in the algorithm each time a witness is encountered.

Proof of (ii):  Suppose for a given $r >0$, $b=p \not \in C_{A,r}$. There exists   $p' \in C_{A,r}$ that does not admit a strict pivot. From Proposition  \ref{prop0} and Cauchy Schwarz inequality,
\begin{equation} \label{bound1}
r \Vert c \Vert < (p-p')^T p  \leq \Vert p - p' \Vert \cdot \Vert b \Vert.
\end{equation}
Given any $r \geq r_0$ in Algorithm \ref{alg2}, $\Vert p - p' \Vert$ is bounded above by the initial gap.  Since initially $p'=0$, $\Vert p -p' \Vert \leq \Vert b \Vert $. Substituting in (\ref{bound1}), $r \Vert c \Vert \leq \Vert b \Vert ^2$. This proves the choice of $r_\varepsilon$  gives (\ref{eqi}).

Proof of (iii):  Given any $p'=Ax'\in C_{A,r}$, using $\Vert p - p' \Vert^2= (p-p')^T(p-p')$ we have,
\begin{equation}
(p-p')^Tp - \Vert p - p' \Vert^2= p'^T(p-p').
\end{equation}
In particular,
\begin{equation} \label{eqiiia}
\bigg | (p-p')^Tp - \Vert p- p' \Vert^2 \bigg | \leq \Vert p'^T(p-p') \Vert.
\end{equation}
Let  $x_*'$ be the minimum-norm solution to $Ax'=p'$. Applying the bound on the minimum-norm solution of this linear system (see (\ref{rstar})) we get,  $\Vert x'_* \Vert \leq \Vert p' \Vert/\sigma_*$. Since $ \Vert p - p' \Vert \leq \Vert b \Vert$, we get
$ \Vert p' \Vert \leq \Vert b \Vert + \Vert b \Vert = 2\Vert b \Vert$.  Hence,
$\Vert x'_* \Vert \leq 2\Vert b \Vert/\sigma_*$. Using this bound and that
$p'^T(p-p')=x_*'^T A^T(p-p')$, from (\ref{eqiiia}) we get
\begin{equation} \label{eqiiib}
\bigg | (p-p')^Tp - \Vert p- p' \Vert^2 \bigg | \leq \Vert x_*' \Vert \cdot \Vert A^T(p-p') \Vert \leq \frac{2 \Vert b \Vert}{\sigma_*} \Vert c \Vert.
\end{equation}
To get the right-hand-side of the above to be less than or equal to $\varepsilon$ it suffices to get $\Vert c \Vert \leq {\sigma_*} \varepsilon/ {2 \Vert b \Vert}$.
From (ii) it follows that when $r \geq r'_\varepsilon$, $\Vert c \Vert$ is properly bounded and hence the first and second inequalities in (\ref{eqii}) are satisfied. Finally, we prove the last inequality in (\ref{eqii}).  From the bound on $x_*$ in (\ref{rstar})  and the definition of
$r'_\varepsilon$ it follows that  $x_* \in C_{A, r'_\varepsilon}$.  Since $p'$ is a witness, it follows from (\ref{deltabound}) in Theorem \ref{thm2p} that
\begin{equation} \label{eqx1}
\Vert p - p' \Vert \leq 2 \Delta = 2 \Vert p - Ax_* \Vert.
\end{equation}
From the first two inequalities in (\ref{eqii}) we may write
\begin{equation}  \label{eqx2}
\Vert p - p' \Vert^2 \geq  (p-p')^Tp  - \varepsilon.
\end{equation}
From (\ref{eqx1}) and (\ref{eqx2}) we have proved the lower bound on $\Delta$.
\end{proof}

\begin{remark}  Some implications of Theorem \ref{finalthm} regarding Algorithm \ref{alg2} are:

(1) Once we have computed an $\varepsilon$-approximate solution to $Ax =b$ in $C_{A,r}$, by using binary search we can compute an approximate solution whose norm is arbitrarily close to $r_*= \Vert x_* \Vert$.

(2) If $Ax=b$ does not admit an $\varepsilon$-approximate solution, Algorithm \ref{alg2} does not  terminate. Hence in practice we have to modify the algorithm so it terminates. A simple modification is to terminate when $r$ exceeds a certain bound.  However, by part (ii) of Theorem \ref{finalthm}, each time the algorithm  computes a point $p' \in C_{A,r}$ which does not admit a strict pivot we can check if $\Vert c \Vert \leq  \varepsilon $.  When $r$ is sufficiently large such $c$ will be at hand. In other words, we can use  $r_\varepsilon$  as an upper bound on $r$ for termination of Algorithm \ref{alg2}.  Then, by part (ii) the algorithm either gives an $\varepsilon$-approximate solution of $Ax=b$ or such approximation to the normal equation $A^TAx=A^Tb$.

(3) If $Ax=b$ does not admit an $\varepsilon$-approximate solution, part (iii) implies there exists $r$ such that both $\Vert c \Vert \leq \varepsilon$ and
$\big | (p-p')^Tp  - \Vert p - p' \Vert^2 \big | \leq \varepsilon$. When this happens  and $(p-p')^Tp$ is not too small, we can terminate the algorithm with the assurance that $Ax=b$ does not admit an $\varepsilon$-approximate solution, or even the assurance that it is unsolvable.

In summary, despite its simplicity, Algorithm \ref{alg2} is not only capable of computing an $\varepsilon$-approximate solution  to $Ax=b$ when such solution exists, but with simple modifications it can compute an $\varepsilon$-approximate solution to the normal equation. Also, it can compute such an approximate solution and also detect unsolvability of $Ax=b$, even place a lower bound on $\Delta$.  Additionally, the algorithm can compute an approximate solution $x'$ to $x_*$, the minimum-norm least-squares solution, where $\Vert x'\Vert$ is as close to $\Vert x_* \Vert$ as desired.   Finally, Algorithm \ref{alg2} can be modified in different ways, e.g. using strict pivots rather than just pivot or alternating between the two and only use a pivot in applications that may need it, e.g. in proving part (iii) of Theorem \ref{finalthm}.
\end{remark}

\section{Computational Results} \label{sec4}
In this section, we present computational results on the Triangle Algorithm, as well as compare the algorithm to the widely used state-of-the-art algorithm BiCGSTAB for solving square linear systems $Ax=b$. The dimension of the matrix $A$ ranges between $100$ and $2000$. We compare both algorithms in three different settings: general, low-rank, and ill-conditioned linear systems explained below.
\begin{itemize}
\item \textbf{General Linear Systems:} In the general random linear system setting, the matrix $A$ is entry-wise randomly generated using two different distributions: Uniform and Gaussian. In particular, the entries are chosen such that $A$ is dense.

\item \textbf{Low-rank Linear Systems:} To generate a matrix $A$ with low rank, we first generate a random matrix similar to the general case, but we run SVD on $A$ and truncate the last $50 \%$ of the singular values to $0$.

\item \textbf{Ill-conditioned Linear Systems:} To generate a matrix $A$ that is ill-conditioned, we first generate a random matrix similar to the general case and run SVD on $A$, then set $50 \%$ of the singular values to $0.001$.
\end{itemize}

Having generated $A$ using one of the three schemes, we randomly generate a solution $x$ by the corresponding distribution and compute $b=Ax$. The goal of the algorithms is to recover $x$. The computational experiments are run on MATLAB 2019a. For BiCGSTAB, we use the module {\it BiCGSTABl} provided by MATLAB. The {\it BiCGSTABl} module requires a preconditioner, so we used the incomplete LU factorization. Without the preconditioner, there is no guarantee that BiCGSTAB converges at all. For each of the three settings a new matrix $A$ is used so that the experiments in the settings are independent.

Figures \ref{tfig:2} and \ref{tfig:3} demonstrate the computational results comparing the speed of the Triangle Algorithm to BiCGSTAB using $\varepsilon=.01$ and $\varepsilon=.001$. The $x$-axis represents the dimensions of $A$ and the $y$-axis represents the running time in seconds. As shown by both figures, the  Triangle Algorithm outperforms BiCGSTAB in speed for both Uniform and Gaussian random matrices in almost all cases. For both distributions, the Triangle Algorithm performs better in the low-rank and ill-conditioned cases and the runtime of the Triangle Algorithm does not increase drastically even when the dimension of $A$ become large.

\begin{figure}[!htb]
\centering
\subfloat[]{\label{fig:6}\includegraphics[width=.31\linewidth]{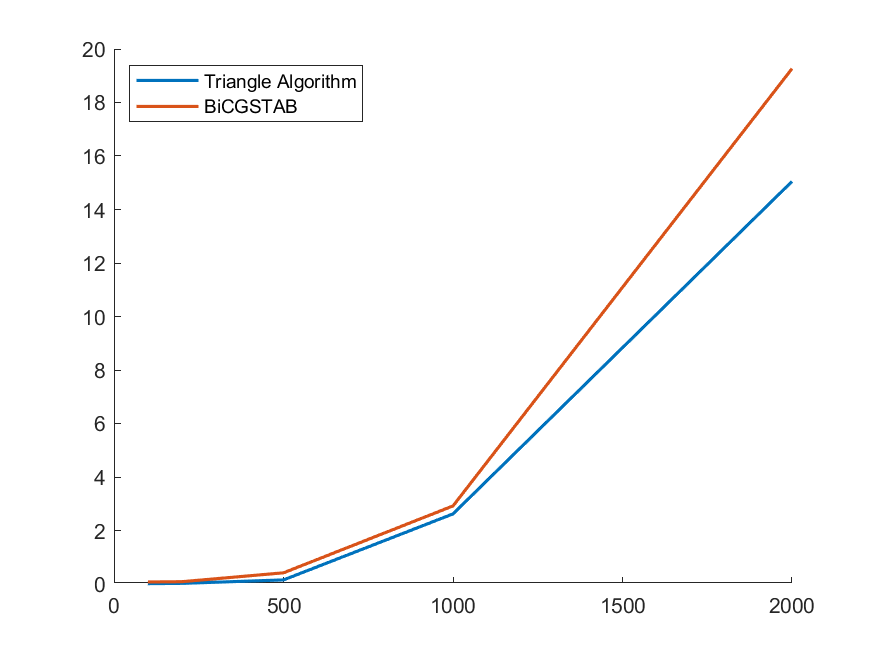}}
\subfloat[]{\label{fig:7}\includegraphics[width=.31\linewidth]{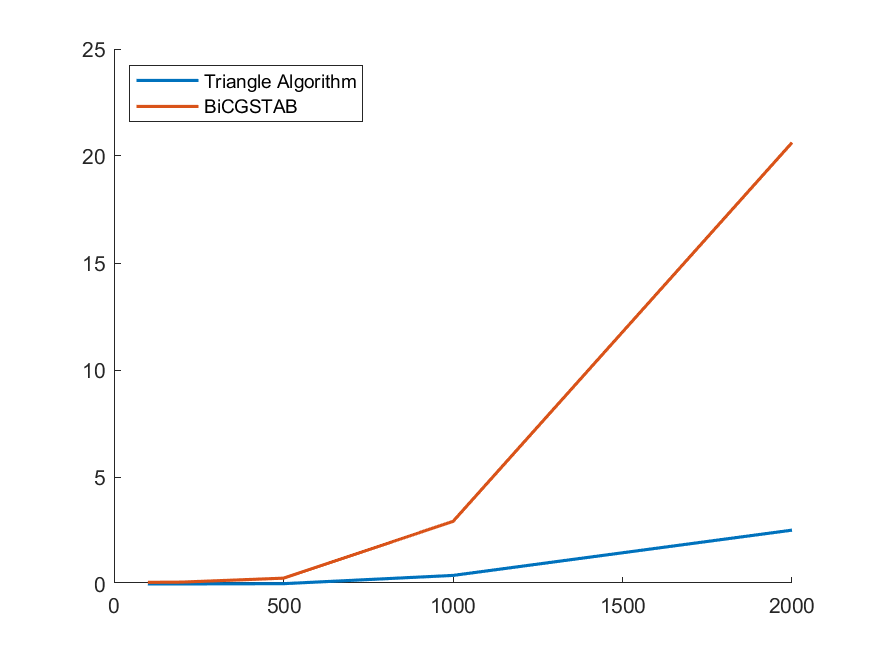}}
\subfloat[]{\label{fig:8}\includegraphics[width=.31\linewidth]{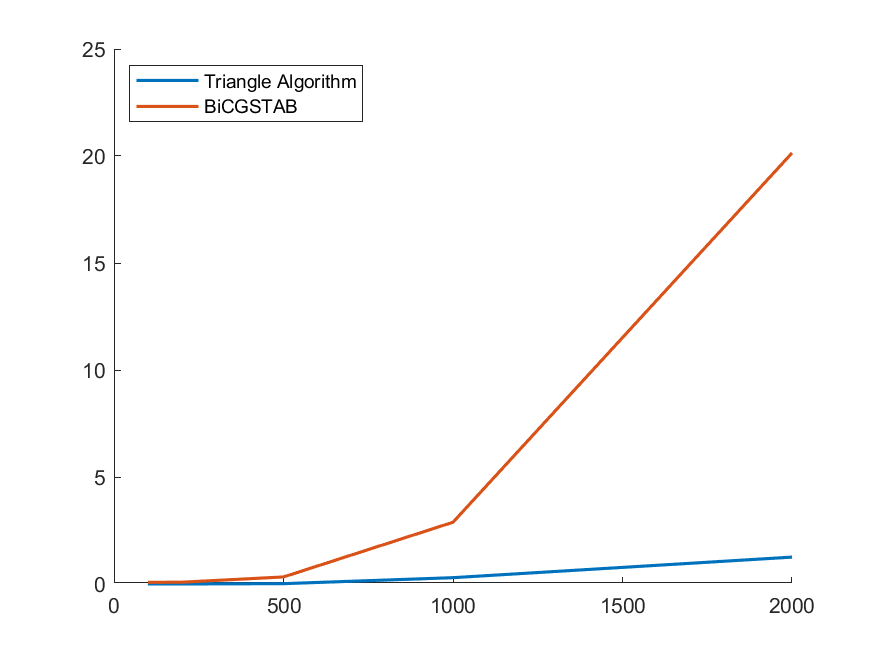}}
\\
\subfloat[]{\label{fig:9}\includegraphics[width=.31\linewidth]{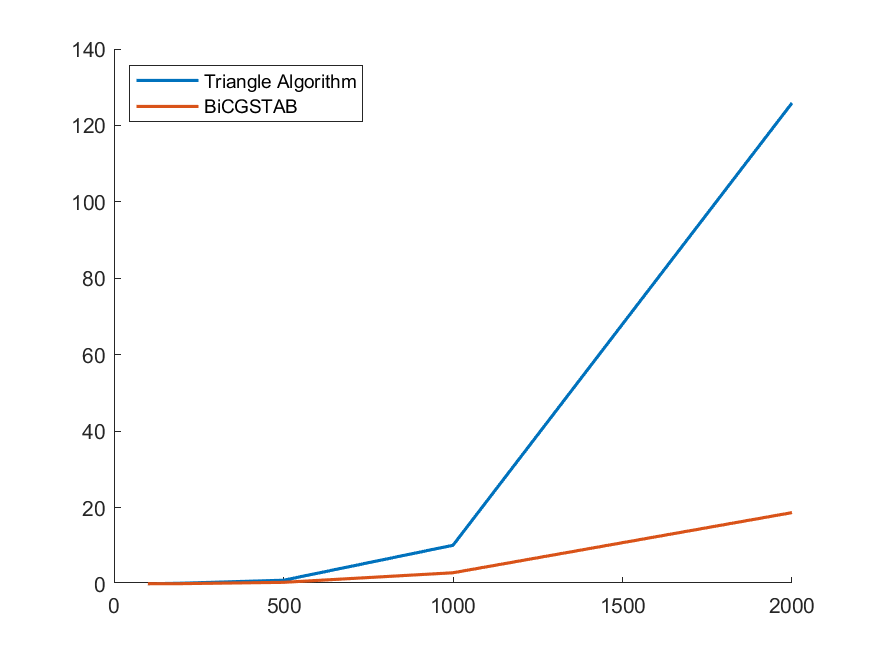}}
\subfloat[]{\label{fig:10}\includegraphics[width=.31\linewidth]{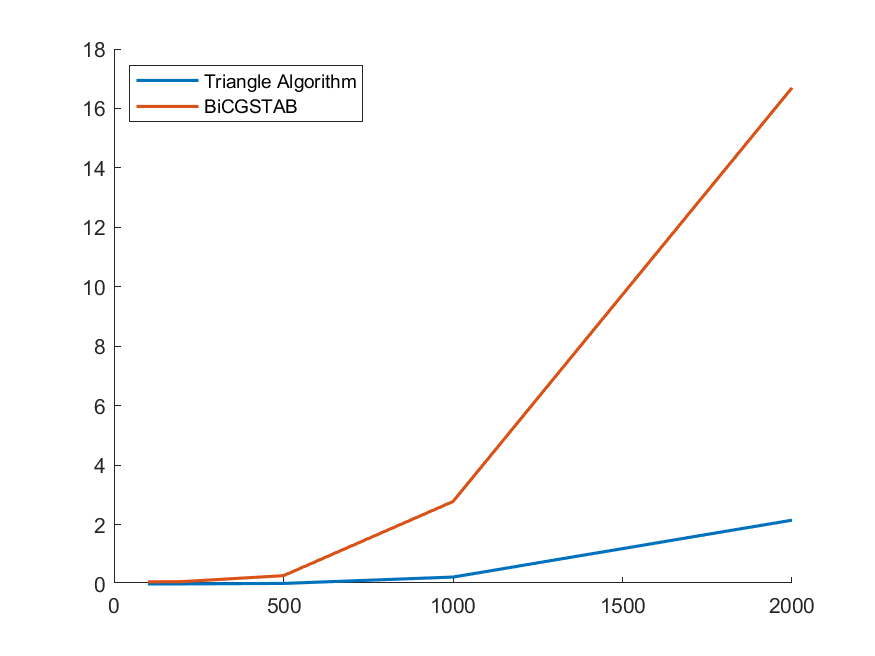}}
\subfloat[]{\label{fig:11}\includegraphics[width=.31\linewidth]{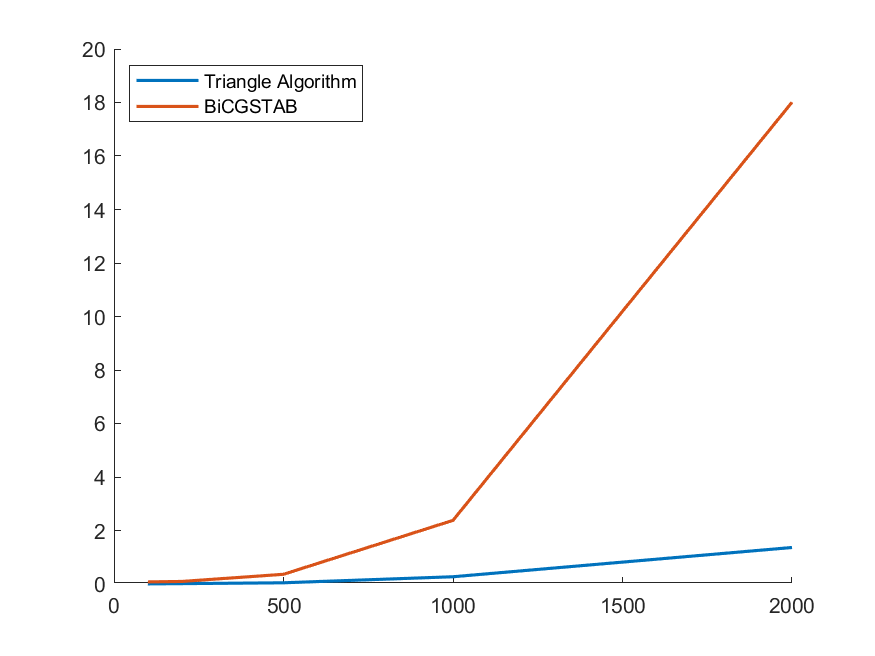}}
\caption{Figures \ref{fig:6}, \ref{fig:7}, \ref{fig:8}, \ref{fig:9}, \ref{fig:10}, and \ref{fig:11} show the plot comparing runtimes between Triangle Algorithm and BiCGSTAB for uniformly random matrices when the linear system is solvable in the regular, low-rank, and ill-conditioned cases respectively. Figures \ref{fig:6}, \ref{fig:7}, and \ref{fig:8} show results for $\varepsilon=0.01$. Figures \ref{fig:9}, \ref{fig:10}, and \ref{fig:11} show results for $\varepsilon=0.001$.}
\label{tfig:2}
\end{figure}

\begin{figure}[htb]
\centering
\subfloat[]{\label{fig:12}\includegraphics[width=.3\linewidth]{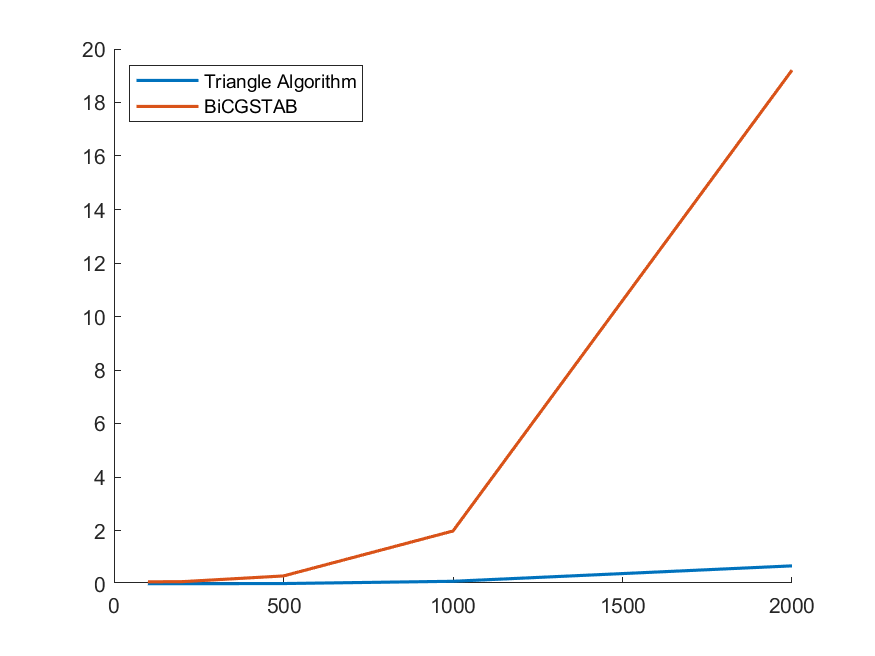}}
\subfloat[]{\label{fig:13}\includegraphics[width=.3\linewidth]{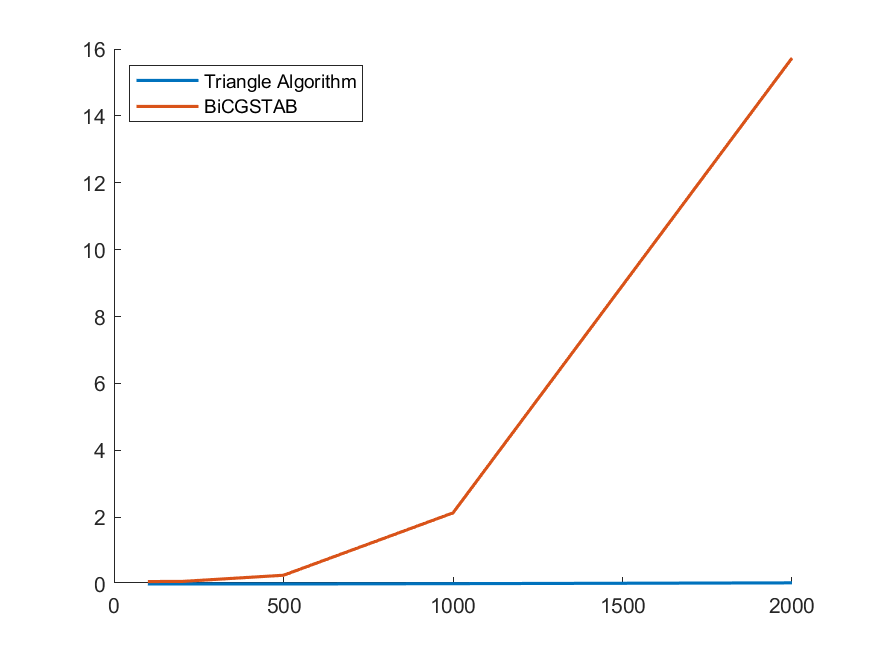}}
\subfloat[]{\label{fig:14}\includegraphics[width=.3\linewidth]{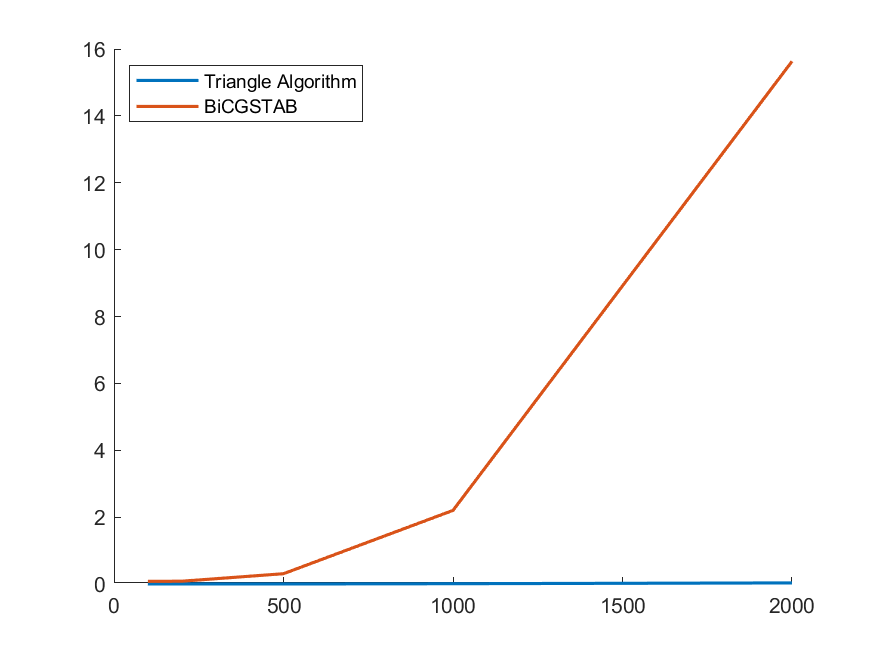}}
\\
\subfloat[]{\label{fig:15}\includegraphics[width=.3\linewidth]{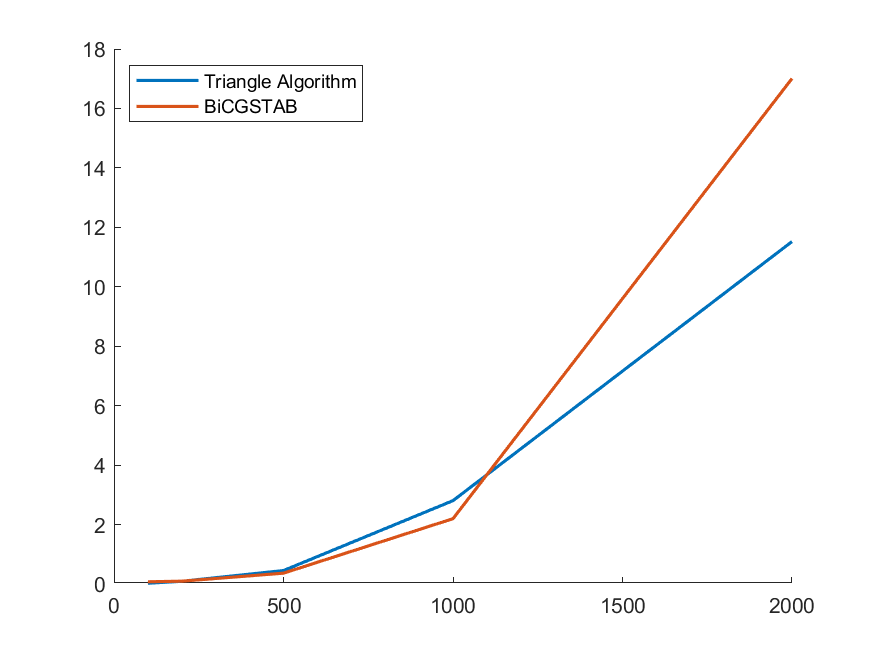}}
\subfloat[]{\label{fig:16}\includegraphics[width=.3\linewidth]{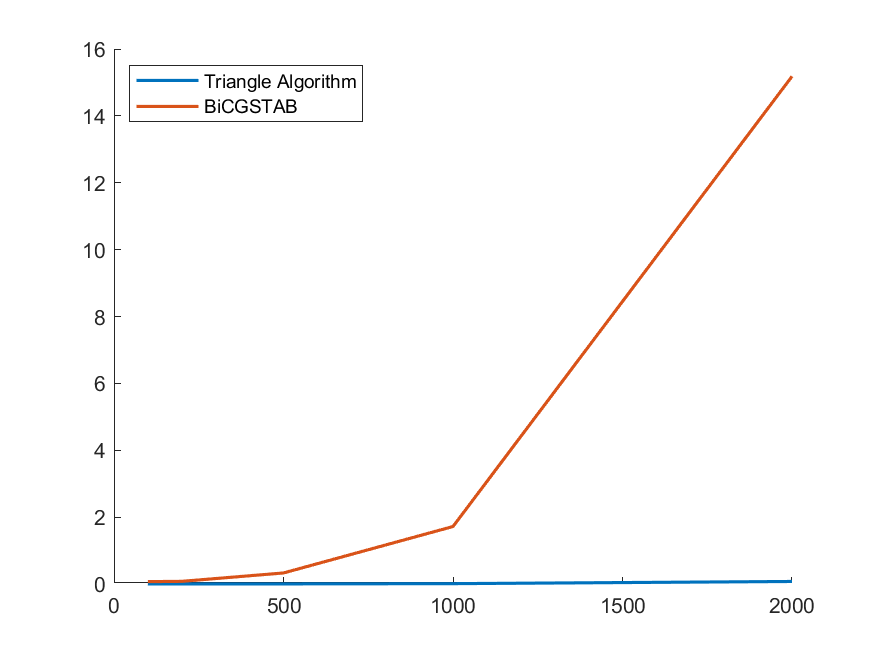}}
\subfloat[]{\label{fig:17}\includegraphics[width=.3\linewidth]{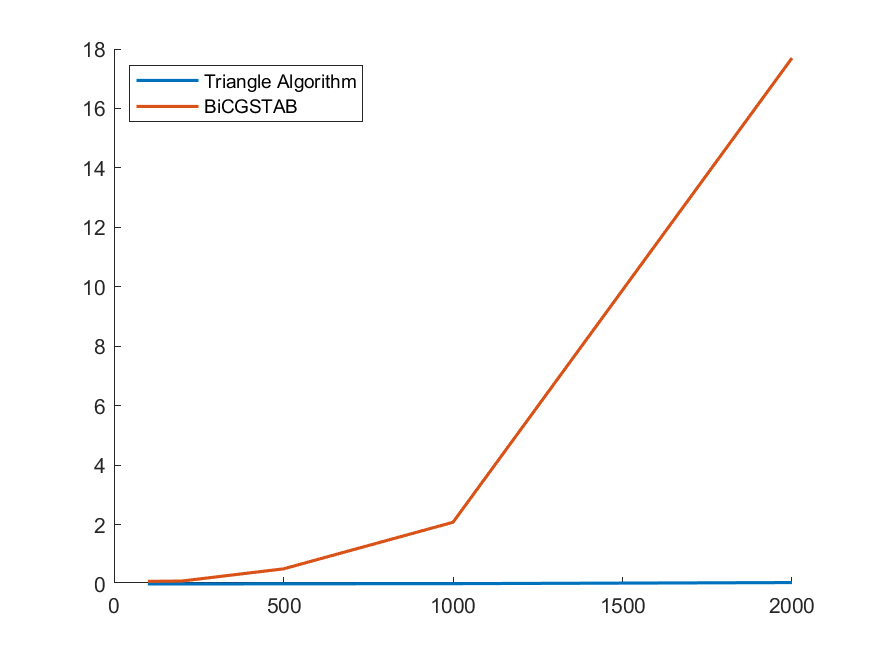}}
\caption{Figures \ref{fig:12}, \ref{fig:13}, \ref{fig:14}, \ref{fig:15}, \ref{fig:16}, and \ref{fig:17} show the plot comparing runtimes between Triangle Algorithm and BiCGSTAB for Gaussian random matrices in the regular, low-rank, and ill-conditioned cases respectively. Figures \ref{fig:12}, \ref{fig:13}, and \ref{fig:14} show results for $\varepsilon=0.01$. Figures \ref{fig:15}, \ref{fig:16}, and \ref{fig:17} show results for $\varepsilon=0.001$.}
\label{tfig:3}
\end{figure}

We also ran the Triangle Algorithm on matrices of various dimensions using different values of $\varepsilon$ . Figures 5 and 6 show plots of runtimes for different values of $\varepsilon$ for general uniform random matrices and low-rank uniform random matrices of dimensions 500, 1000, and 2000. The $x$-axis represents ${1}/{\varepsilon}$ and the $y$-axis represents running time. While Figure 5 shows that the runtimes of the Triangle Algorithm for general uniform random matrices have some dependence on ${1}/{\varepsilon}$, Figure 6 shows that changes to ${1}/{\varepsilon}$ did not significantly affect the runtimes of the Triangle Algorithm for low-rank uniform random matrices. These results further strengthen the case for using Triangle Algorithm to solve $Ax=b$, especially when $A$ might be of low-rank.

\begin{figure}[!htb]
\centering
\subfloat[]{\label{fig:20}\includegraphics[width=.3\linewidth]{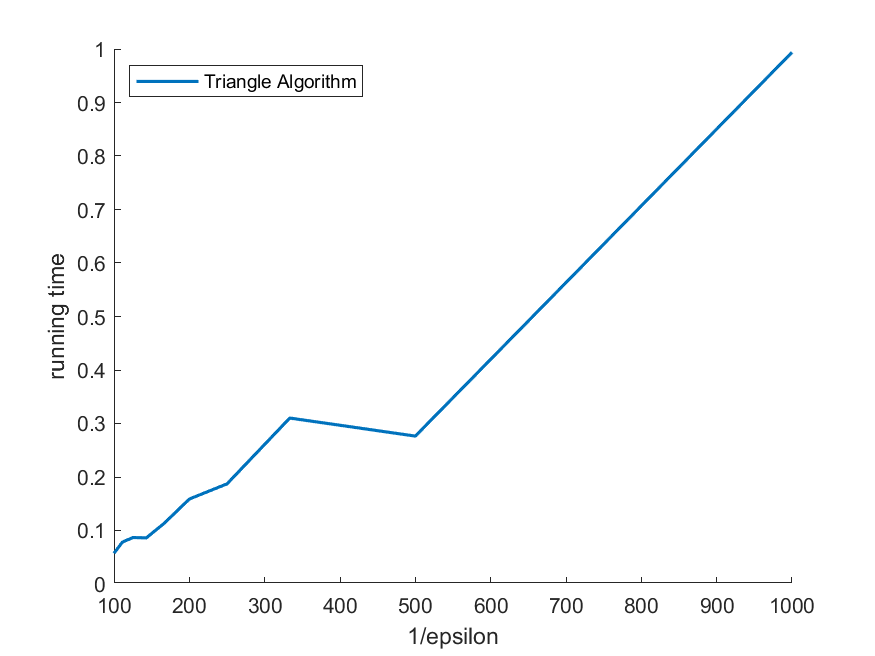}}
\subfloat[]{\label{fig:21}\includegraphics[width=.3\linewidth]{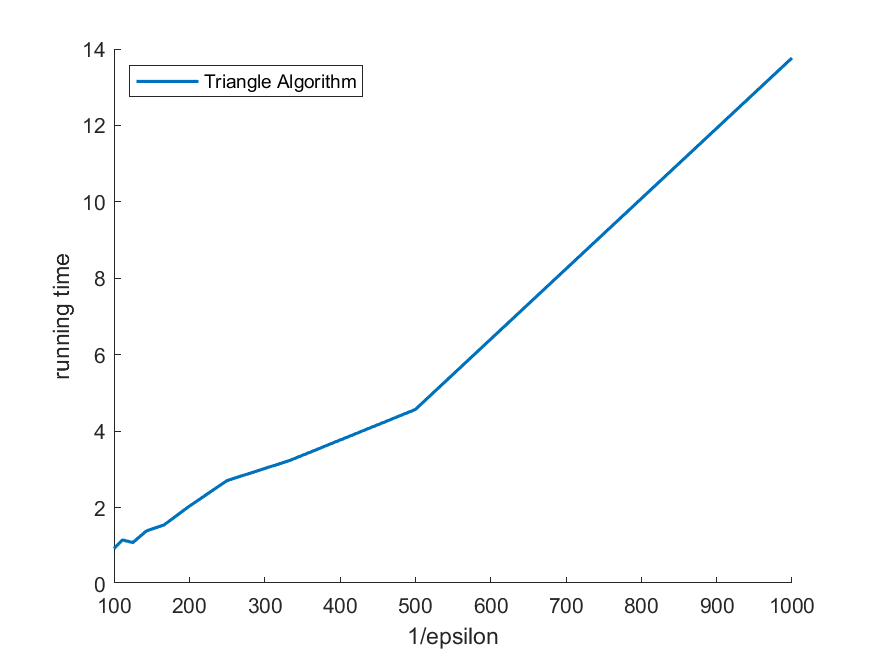}}
\subfloat[]{\label{fig:22}\includegraphics[width=.3\linewidth]{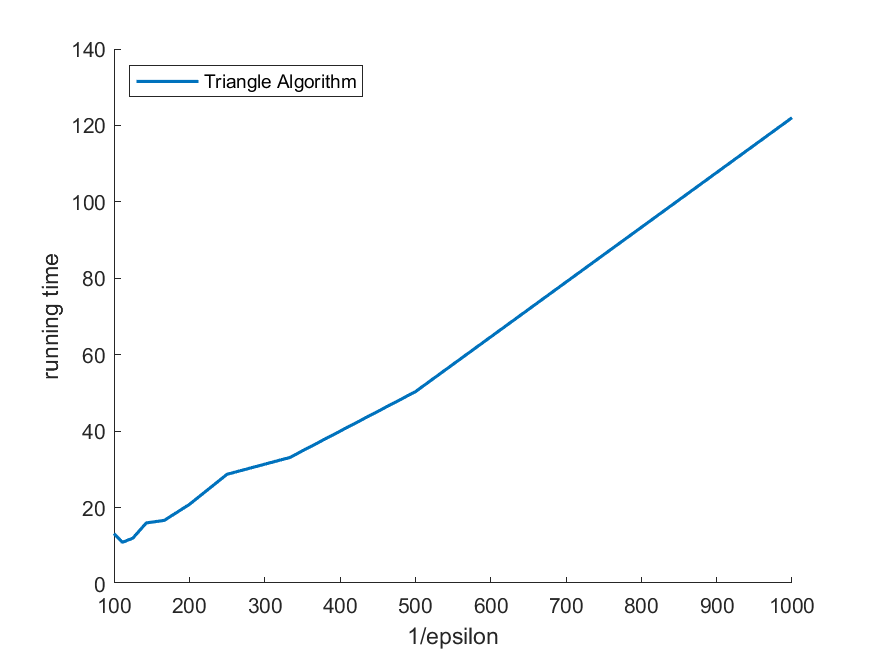}}

\caption{Figures \ref{fig:20},  \ref{fig:21}, \ref{fig:22} show the plot of running time for the Triangle Algorithm using different choices of $\varepsilon$ on general uniform random matrices of dimensions 500, 1000, and 2000 respectively. The $x$-axis represents $\frac{1}{\varepsilon}$ and the $y$-axis represents running time.}
\label{epsfig:2}
\end{figure}
\begin{figure}[htb]
\centering
\subfloat[]{\label{fig:25}\includegraphics[width=.3\linewidth]{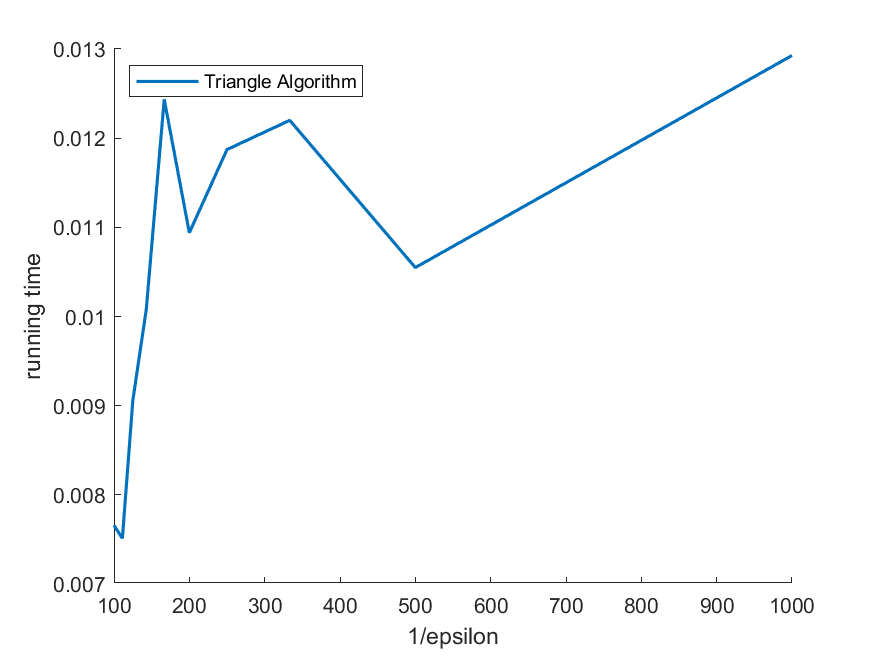}}
\subfloat[]{\label{fig:26}\includegraphics[width=.3\linewidth]{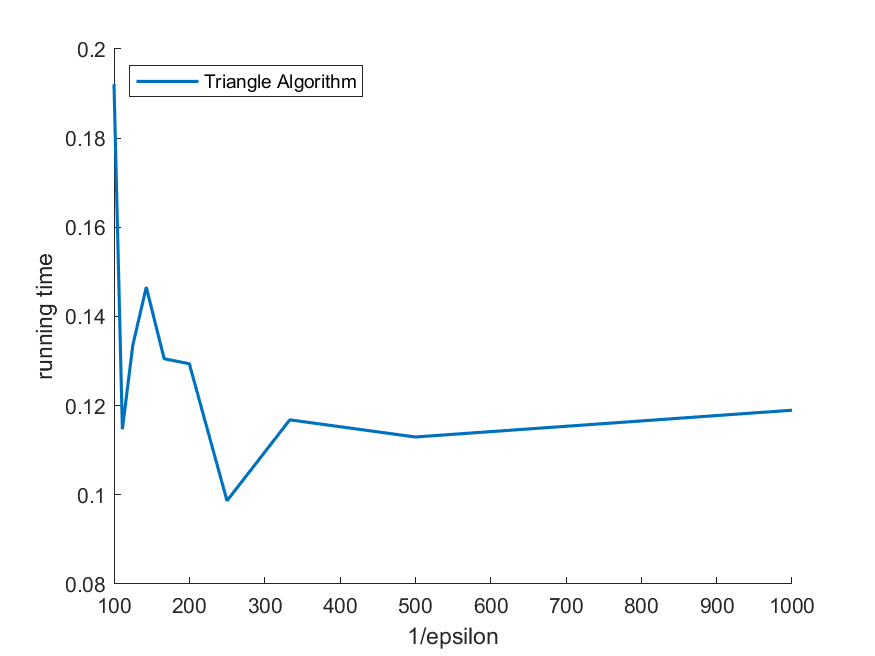}}
\subfloat[]{\label{fig:27}\includegraphics[width=.3\linewidth]{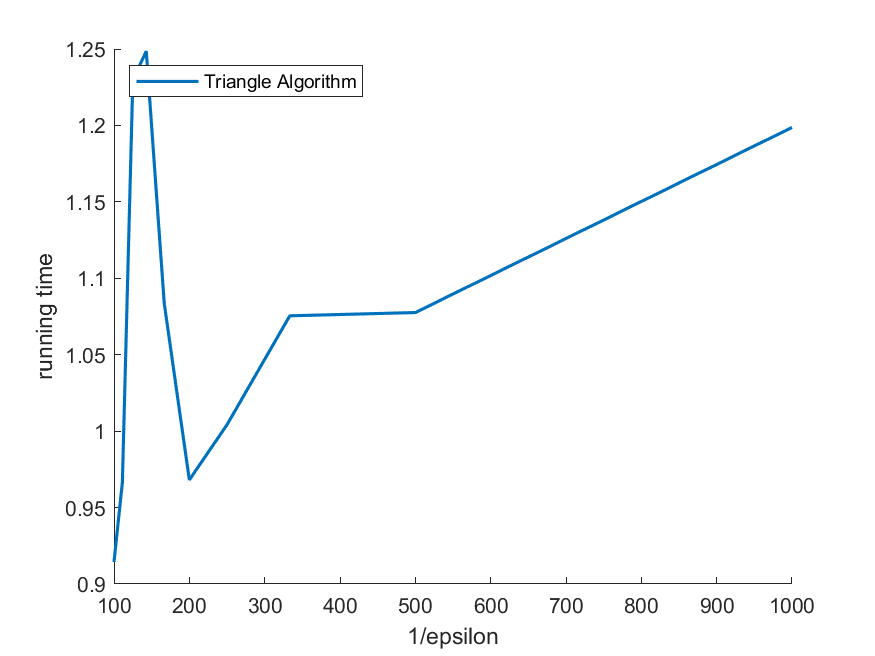}}

\caption{Figures \ref{fig:25},  \ref{fig:26}, \ref{fig:27} show the plot of running time for the Triangle Algorithm using different choices of $\varepsilon$ on low-rank uniform random matrices of dimensions 500, 1000, and 2000 respectively. The $x$-axis represents $\frac{1}{\varepsilon}$ and the $y$-axis represents running time.}
\label{epsfig:3}
\end{figure}

\section*{Concluding Remarks}
Based on the geometrically inspired Triangle Algorithm for testing membership in a compact convex set, in this article we have developed an easy-to-implement version of the Triangle Algorithm for approximating the solution, or the minimum-norm least-squares solution to a linear system $Ax=b$, where $A$ is an $m \times n$ matrix. An important feature  of the Triangle Algorithm is that there are no constraints on $A$ such as invertibility, full-rankness, etc. In this article we have also compared computational results for solving square linear systems via the Triangle Algorithm and BiCGSTAB, the state-of-the-art algorithm for solving such systems. We have found the Triangle Algorithm to be extremely competitive.  While Triangle Algorithm outperforms BiCGSTAB in almost every setting of the experiment, it performs particularly well in the low-rank and ill-conditioned cases for solvable linear systems. The Triangle Algorithm can detect unsolvability of a linear system and go on approximating the least-squares or minimum-norm least-squares solution. In contrast, when a square linear system is not solvable BiCGSTAB does not converge. Based on these results we conclude the Triangle Algorithm is a powerful algorithm for solving linear systems and can be widely used in practice. Despite the fact that our analysis is with respect to real input, all the analysis can be extended to complex input. This is because the Triangle Algorithm is based on Euclidean distance, and over complex domain, the notion of distance for a complex vector is defined accordingly. Finally, as a matter of comparison we have also tested the Triangle Algorithm against the steepest descent method \cite{shew} and have found the steepest descent method too slow to compete. In our forthcoming work we will further consider the application of the Triangle Algorithm in solving least-squares problem and will contrast it with the steepest descent method and Nesterov's accelerated gradient method \cite{Nest}.

\bibliographystyle{plain}
\bibliography{biblio1}
\end{document}